\documentclass[11pt,reqno]{amsart}

\usepackage[T1]{fontenc}
\usepackage{lmodern}
\usepackage{microtype}
\usepackage{mathtools}
\usepackage{amssymb}
\usepackage{enumitem}
\usepackage{booktabs}
\usepackage{aliascnt}
\usepackage[margin=1.03in]{geometry}
\usepackage[dvipsnames]{xcolor}
\usepackage[colorlinks=true,
            linkcolor=MidnightBlue,
            citecolor=MidnightBlue,
            urlcolor=MidnightBlue,
            pdfauthor={Chenxiao Tian},
            pdftitle={Frankl's Union-Closed Sets Conjecture for Families of Height at Most Four and Structural Reductions at Height Five},
            pdfsubject={Union-closed families of bounded height},
            pdfkeywords={union-closed sets conjecture, Frankl's conjecture, bounded height, minimal counterexample, critical join-cover}]{hyperref}

\numberwithin{equation}{section}

\newtheorem{theorem}{Theorem}[section]

\newaliascnt{lemma}{theorem}
\newtheorem{lemma}[lemma]{Lemma}
\aliascntresetthe{lemma}

\newaliascnt{proposition}{theorem}
\newtheorem{proposition}[proposition]{Proposition}
\aliascntresetthe{proposition}

\newaliascnt{corollary}{theorem}
\newtheorem{corollary}[corollary]{Corollary}
\aliascntresetthe{corollary}

\newaliascnt{claim}{theorem}

\aliascntresetthe{claim}

\theoremstyle{definition}
\newaliascnt{definition}{theorem}
\newtheorem{definition}[definition]{Definition}
\aliascntresetthe{definition}

\newaliascnt{problem}{theorem}
\newtheorem{problem}[problem]{Problem}
\aliascntresetthe{problem}

\theoremstyle{remark}
\newaliascnt{remark}{theorem}
\newtheorem{remark}[remark]{Remark}
\aliascntresetthe{remark}

\usepackage[nameinlink,capitalise,noabbrev]{cleveref}

\crefname{equation}{equation}{equations}
\crefname{lemma}{lemma}{lemmas}
\crefname{proposition}{proposition}{propositions}
\crefname{theorem}{theorem}{theorems}
\crefname{corollary}{corollary}{corollaries}
\crefname{problem}{problem}{problems}
\crefname{remark}{remark}{remarks}

\newcommand{\F}{\mathcal{F}}
\newcommand{\A}{\mathcal{A}}
\newcommand{\G}{\mathcal{G}}
\newcommand{\Hh}{\mathcal{H}}
\newcommand{\W}{\mathcal{W}}
\newcommand{\J}{\mathcal{J}}
\newcommand{\1}{\mathbf{1}}
\newcommand{\set}[1]{\left\{#1\right\}}
\newcommand{\abs}[1]{\left|#1\right|}
\newcommand{\join}{\mathbin{\vee}}

\title[Height Four and Structural Reductions at Height Five]
{Frankl's Union-Closed Sets Conjecture:\\
Height Four and Structural Reductions at Height Five}

\author{Chenxiao Tian}
\address{Princeton University, Princeton, New Jersey 08544, USA}
\email{ct3471@alumni.princeton.edu}
\date{August 2026}


\begin{document}

\begin{abstract}
Let $h(\F)$ denote the largest number of members in a strict inclusion chain of a finite union-closed family $\F$.  We work in the empty-set-free normalization, in which Frankl's conjecture asks for an element contained in strictly more than half of the members.  We prove the conjecture for $h(\F)\le4$.  The proof is governed by an equality case: a smallest counterexample has $2t$ members, at least three coordinates of frequency $t$, and $t\ge2n-1$, where $n=\lvert\bigcup\F\rvert$; a critical pair then forces its double-avoidance fiber to be the unique extremal height-two union-closed family, while a third critical coordinate yields the contradictory identity $t=2n-2$.

For empty-set-free height five we do not claim a complete proof.  Instead, we derive a structural theorem for a smallest counterexample.  Every critical coordinate $x$ has the coatom $U\setminus\{x\}$ in the family.  Every critical pair satisfies a full-top/large-fiber dichotomy, and every non-full pair gives an exact three-trace normal form.  Projection along a critical coordinate produces a matching-defect witness inequality.  We also introduce the critical join-cover number, the least number of join-irreducible members whose union contains all critical coordinates, and prove that it is either three or four: the a priori possible five-cover branch is excluded by a Boolean-fiber decomposition and a convexity argument.  If a minimal four-cover does not cover the universe, exactly one noncritical coordinate lies outside it and the family has a stable two-fiber normal form.  The remaining height-five branches and the transfer inequalities needed to eliminate them are stated explicitly.
\end{abstract}

\maketitle

\section{Introduction}

All families in this paper are finite and nonempty, and all member sets are finite.  A family \(\A\) is \emph{union-closed} if
\[
A,B\in\A \quad\Longrightarrow\quad A\cup B\in\A.
\]
Frankl's union-closed sets conjecture asserts that every nontrivial finite union-closed family contains an element belonging to at least half of its members.  We refer to the survey of Bruhn and Schaudt~\cite{BruhnSchaudt} for background and to Poonen~\cite{Poonen} for a foundational structural treatment.

The order-theoretic parameter considered here is the height
\[
H(\A)=\max\set{s:\ A_1\subsetneq A_2\subsetneq\cdots\subsetneq A_s,
\ A_i\in\A}.
\]
The height viewpoint was introduced in~\cite{TianHeight}; a general chain-condition proof for families of height at most three was subsequently obtained by Colbert~\cite{Colbert}.  Bouchard's averaging result treats separating height-four families under the additional hypothesis that a smallest irredundant cover of the small-set subfamily has at most two members~\cite{Bouchard}.  Here we prove the height-four statement without that auxiliary hypothesis and then identify a concrete residual obstruction at height five.

\subsection*{The convention and the parity issue}
Throughout the main argument, \(\F\) is a nonempty finite union-closed family of distinct nonempty sets, so \(\varnothing\notin\F\).  We seek an element occurring in \emph{strictly more} than \(\abs{\F}/2\) members.  This is equivalent to the standard formulation applied to
\[
\F^+=\F\cup\set{\varnothing},
\]
because frequencies do not change while the number of members increases by one.  The height changes as well: \(H(\F^+)=H(\F)+1\).

The distinction between the two conventions is not cosmetic.  If \(\abs{\F}=2t\) and two elements \(x,y\) each occur exactly \(t\) times, then the four-cell count satisfies
\[
N_{00}=N_{11}.
\]
After adjoining \(\varnothing\), the empty set enters the \(00\)-cell and one instead has
\[
N^{+}_{00}=N_{11}+1.
\]
Mixing these identities produces a spurious one-unit contradiction.  All counting below is carried out in the empty-set-free even family \(\F\).

Our first main theorem is the following.

\begin{theorem}[Height four]\label{thm:main-height4}
Let \(\F\) be a nonempty finite union-closed family of nonempty sets.  If \(H(\F)\le 4\), then there is an element \(x\in\bigcup\F\) such that
\[
\abs{\F_x}>\frac{\abs{\F}}2,
\qquad
\F_x:=\set{A\in\F:x\in A}.
\]
\end{theorem}

\begin{corollary}[Height five with a bottom element]\label{cor:height5-bottom}
Let \(\A\) be a finite nontrivial union-closed family.  If \(\varnothing\in\A\) and \(H(\A)\le5\), then some element belongs to at least half of the members of \(\A\).
\end{corollary}

\begin{proof}
Set \(\F=\A\setminus\set{\varnothing}\).  Every chain in \(\F\) can be extended downward by \(\varnothing\), so \(H(\F)=H(\A)-1\le4\).  By \cref{thm:main-height4}, some element occurs in more than \(\abs{\F}/2\) members of \(\F\).  Its frequency is unchanged in \(\A\), and integrality gives frequency at least \(\abs{\A}/2\).
\end{proof}

At empty-set-free height five we prove a package of necessary conditions for a smallest counterexample.  The most useful pieces are the critical-coatom theorem, a dichotomy for every pair of critical elements, a trace-layer normal form for the nonfull case, and an exact projection formula.  The next theorem records the logical output in one place; all notation is defined in the corresponding sections.

\begin{theorem}[Structure of a smallest height-five counterexample]\label{thm:height5-summary}
Assume that an empty-set-free counterexample with height at most five exists, and choose one with the minimum number of members and then the minimum ground-set size.  Write \(U=\bigcup\F\), \(n=\abs{U}\), and \(\abs{\F}=2t\).  Then:
\begin{enumerate}[label=\textup{(\roman*)},leftmargin=2.5em]
\item at least three coordinates have frequency \(t\), every coordinate has frequency at most \(t\), and \(t\ge2n-1\);
\item for every critical coordinate \(x\), the largest member avoiding \(x\) is the coatom \(U\setminus\set{x}\);
\item for every distinct critical pair \(x,y\), either the union of all members avoiding both is \(U\setminus\set{x,y}\), or the double-avoidance fiber has at least \(\lceil(t+1)/2\rceil\ge n\) members;
\item in the second alternative, the family has the three-trace normal form of \cref{thm:trace-normal}, together with the exact layer counts of \cref{prop:layer-count};
\item for every critical coordinate \(p\), the deletion projection gives the exact witness inequality \cref{eq:projection-witness}, and its matching fiber has size at least three;
\item the critical join-cover number defined in \cref{def:critical-cover} belongs to \(\{3,4\}\);
\item if a minimal four-member critical join-cover has union \(V\ne U\), then \(U\setminus V\) is a single noncritical coordinate and \(\F\) has the stable two-fiber normal form of \cref{thm:four-cover-normal}.
\end{enumerate}
\end{theorem}

\begin{remark}[Status of height five]\label{rem:height5-status}
\Cref{thm:height5-summary} is a conditional structure theorem: it describes every smallest counterexample if one exists, but it does not settle the unrestricted empty-set-free height-five case.  The remaining branches and two explicit transfer targets are recorded in \cref{sec:remaining}; neither target is assumed in any proof.
\end{remark}

The minimal-counterexample arguments descend from Norton and Sarvate~\cite{NortonSarvate}, Roberts and Simpson~\cite{RobertsSimpson}, Lo Faro~\cite{LoFaro}, and Hu~\cite{Hu}.  We include the required height-preserving adaptations in full.

\section{Minimal counterexamples under a height bound}

Let
\[
U=\bigcup\F,
\qquad n=\abs{U},
\qquad
f(x)=\abs{\F_x}\quad(x\in U).
\]
A family is \emph{separating} if \(\F_x\ne\F_y\) for all distinct \(x,y\in U\).

Fix an integer \(h\ge 1\), and suppose that a counterexample with \(H(\F)\le h\) exists.  Choose one with \(\abs{\F}\) minimum and, subject to that, with \(\abs{U}\) minimum.  We call such a family an \emph{\(h\)-minimal counterexample}.

\begin{lemma}[Separating reduction]\label{lem:separating}
Every \(h\)-minimal counterexample is separating.
\end{lemma}

\begin{proof}
If distinct coordinates \(x,y\) have the same occurrence vector, delete \(y\) from every member of \(\F\).  Because \(x\) and \(y\) always occur together, distinct members remain distinct after the deletion.  The resulting family is union-closed, has the same number of members, has height no larger than \(H(\F)\), and remains a counterexample because the frequency of every surviving coordinate is unchanged and the frequency of \(x\) equals that of the deleted coordinate.  Its ground set is smaller, contradicting the choice of \(\F\).
\end{proof}

\subsection{Parity and critical elements}

\begin{lemma}[Minimal parity and three critical coordinates]\label{lem:critical-three}
Let \(\F\) be an \(h\)-minimal counterexample.  Then
\[
\abs{\F}=2t
\]
for some positive integer \(t\), every element has frequency at most \(t\), and the critical set
\[
K:=\set{x\in U:f(x)=t}
\]
has at least three elements.
\end{lemma}

\begin{proof}
Suppose first that \(\abs{\F}=2t+1\).  Since \(\F\) is a counterexample, every frequency is at most \(t\).  Let \(A_0\) be an inclusion-minimal member of \(\F\).  The family \(\F\setminus\set{A_0}\) is union-closed: if the union of two remaining members were \(A_0\), then both would be proper members of \(\F\) contained in \(A_0\), contrary to minimality.  The smaller family has \(2t\) members and no frequency exceeding \(t\), so it is again a counterexample of height at most \(h\), a contradiction.  Thus \(\abs{\F}=2t\), and necessarily \(f(x)\le t\) for all \(x\).  A counterexample cannot have \(t=1\): with two distinct nonempty members, union-closure forces one to contain the other, and every element of the smaller member then has frequency two.  Hence \(t\ge2\).

Again choose an inclusion-minimal member \(A_0\).  The family \(\F\setminus\set{A_0}\) has \(2t-1\) members and, by minimality of \(\F\), has an element occurring in at least \(t\) of them.  Hence there is an element \(x\) with
\[
f(x)=t,
\qquad x\notin A_0.
\]
Let \(A_x\) be inclusion-minimal among the members containing \(x\).  The family
\(
\F\setminus\set{A_0,A_x}
\)
is union-closed.  Indeed, a missing union could only equal \(A_0\) or \(A_x\); the first possibility contradicts minimality of \(A_0\), and in the second one of the two unionands would be a proper member containing \(x\) inside \(A_x\).  The remaining family has \(2t-2\) members, so it contains an element in at least \(t\) members.  Consequently there is a critical element \(y\ne x\) with
\[
y\notin A_0\cup A_x.
\]
Choose \(A_y\) inclusion-minimal among members containing \(y\).  The family
\(
\F\setminus\set{A_x,A_y}
\)
is union-closed by the same argument and has \(2t-2\) members.  It therefore contains an element in at least \(t\) members.  This element cannot be \(x\) or \(y\), since a member containing each of them was deleted.  We obtain a third critical element \(z\), proving \(\abs{K}\ge3\).
\end{proof}

\subsection{Maximal avoiding sets}

For each \(a\in U\), define
\begin{equation}\label{eq:def-Ca}
C_a:=\bigcup\set{A\in\F:a\notin A}.
\end{equation}
Since \(f(a)\le t<2t\), the subfamily in \cref{eq:def-Ca} is nonempty.  Finite union-closure gives \(C_a\in\F\).

\begin{lemma}[Avoiding-set properties]\label{lem:Ca-properties}
Let \(\F\) be an \(h\)-minimal counterexample.
\begin{enumerate}[label=\textup{(\roman*)},leftmargin=2.3em]
\item \(C_a\) is the largest member of \(\F\) avoiding \(a\), and \(a\notin C_a\).
\item If \(a\ne b\), then \(C_a\ne C_b\).
\item If \(x\in K\) and \(a\ne x\), then \(x\in C_a\).
\end{enumerate}
\end{lemma}

\begin{proof}
Part (i) follows from the definition.  For (ii), if a member contains \(a\) but not \(b\), then it is contained in \(C_b\), so \(a\in C_b\); this is incompatible with \(C_a=C_b\), since \(a\notin C_a\).  Separation guarantees that one of the two possible directions exists.

For (iii), suppose \(x\notin C_a\).  Then no member avoiding \(a\) contains \(x\), equivalently \(\F_x\subseteq\F_a\).  Since \(f(x)=t\ge f(a)\), either the inclusion contradicts \(f(a)<t\), or it is equality and contradicts separation.
\end{proof}

\subsection{A height-preserving Lo Faro--Hu bound}

The next estimate is the only global size bound needed in the height-four proof.  We give a self-contained adaptation of Hu's staircase argument~\cite{Hu}; the same estimate is related to the minimal-counterexample bounds of Lo Faro~\cite{LoFaro} and Roberts--Simpson~\cite{RobertsSimpson}.

\begin{lemma}[Hu staircase]\label{lem:hu-staircase}
Let \(\A\) be a finite separating union-closed family on
\(
U=\set{u_1,\dots,u_n}
\), indexed so that
\[
f_{\A}(u_1)\le f_{\A}(u_2)\le\cdots\le f_{\A}(u_n).
\]
For \(i<n\), let \(C_i\) be the union of all members avoiding \(u_i\), and set
\[
\mathcal S=\set{U,C_1,\dots,C_{n-1}}.
\]
Then the members of \(\mathcal S\) are distinct and all contain \(u_n\).  Moreover, if \(\mathcal B\subseteq\A\) is a subfamily all of whose members avoid \(u_n\), and \(\mathcal B\) contains a nonempty set, then among the elements of maximum frequency in \(\mathcal B\) there is an element occurring in exactly \(n-1\) members of \(\mathcal S\).
\end{lemma}

\begin{proof}
If \(i<j\), then \(u_j\in C_i\).  Otherwise every member containing \(u_j\) would also contain \(u_i\), so
\(
\A_{u_j}\subseteq\A_{u_i}
\); the frequency ordering would force equality, contradicting separation.  It follows that all members of \(\mathcal S\) contain \(u_n\), and they are distinct because \(u_j\in C_i\setminus C_j\) whenever \(i<j<n\), while \(U\ne C_i\).

Fix \(i<n\).  The element \(u_i\) belongs to \(U\) and to every \(C_j\) with \(j<i\), while it does not belong to \(C_i\).  If it occurs fewer than \(n-1\) times in \(\mathcal S\), then it also misses some \(C_j\) with \(i<j<n\).  The relation \(u_i\notin C_j\) means that every member containing \(u_i\) also contains \(u_j\); in other words, \(u_j\) dominates \(u_i\).  Iterating strictly increases the index and therefore terminates at an element \(u_k\), with \(k<n\), that occurs exactly \(n-1\) times in \(\mathcal S\).

Choose an element of maximum positive frequency in \(\mathcal B\).  It is not \(u_n\).  Along the domination chain its frequency in \(\mathcal B\) cannot decrease; hence every element reached still has maximum frequency in \(\mathcal B\).  The terminal element has the required \(\mathcal S\)-frequency.
\end{proof}

\begin{proposition}[Restricted minimal-counterexample bound]\label{prop:lower-bound}
If \(\F\) is an \(h\)-minimal counterexample, \(\abs{\F}=2t\), and \(\abs{U}=n\), then
\begin{equation}\label{eq:t-lower}
\boxed{t\ge 2n-1.}
\end{equation}
\end{proposition}

\begin{proof}
Adjoin the empty set and write
\[
\A=\F\cup\set{\varnothing},
\qquad \abs{\A}=2t+1.
\]
Frequencies are unchanged, and \(\A\) is separating by \cref{lem:separating}.  Order the ground elements by frequency and choose the last element \(u_n\) to be critical, so \(f_{\A}(u_n)=t\).  The subfamily
\[
\A_{\overline{u_n}}=\set{A\in\A:u_n\notin A}
\]
has \(t+1\) members and contains \(\varnothing\).

This subfamily satisfies the standard union-closed conclusion.  Indeed, if no element occurred in at least half of its \(t+1\) members, then after removing \(\varnothing\) every frequency would be at most \(\lfloor t/2\rfloor\).  The resulting union-closed family would have exactly \(t\) nonempty members, height at most \(h\), and no strict-majority element.  It would therefore be a counterexample smaller than \(\F\), contrary to minimality.

Apply \cref{lem:hu-staircase} with \(\mathcal B=\A_{\overline{u_n}}\).  There is an element \(a\ne u_n\) that occurs in at least
\(
\lceil (t+1)/2\rceil
\)
members of \(\mathcal B\) and in exactly \(n-1\) members of the staircase \(\mathcal S\).  Every member of \(\mathcal S\) contains \(u_n\), whereas every member of \(\mathcal B\) avoids it, so the two subfamilies are disjoint.  Since \(f_{\A}(a)\le t\),
\[
\left\lceil\frac{t+1}{2}\right\rceil+n-1\le t.
\]
Using \(\lceil (t+1)/2\rceil\ge(t+1)/2\) gives \(t\ge2n-1\).
\end{proof}

\begin{remark}[Classical form of the bound]\label{rem:classical-bound}
For \(\F^+=\F\cup\set{\varnothing}\), the estimate becomes
\[
\abs{\F^+}=2t+1\ge4n-1.
\]
Thus \cref{prop:lower-bound} is exactly the familiar smallest-counterexample lower bound, written in the parity normalization used here.
\end{remark}

\subsection{The two-coordinate parity identity}

For distinct \(x,y\in U\), write
\begin{equation}\label{eq:Nij}
N_{ij}^{xy}
=
\abs{\set{A\in\F:
\1_{x\in A}=i,\ \1_{y\in A}=j}},
\qquad i,j\in\set{0,1}.
\end{equation}
When the pair is fixed, we suppress the superscript.

\begin{lemma}[Critical-pair parity]\label{lem:pair-parity}
If \(x,y\in K\) are distinct, then
\begin{equation}\label{eq:N00=N11}
\boxed{N_{00}=N_{11}.}
\end{equation}
Moreover,
\begin{equation}\label{eq:N11-lower}
N_{11}\ge n-1.
\end{equation}
\end{lemma}

\begin{proof}
The two marginal identities and the total count are
\[
N_{10}+N_{11}=t,
\qquad
N_{01}+N_{11}=t,
\qquad
N_{00}+N_{01}+N_{10}+N_{11}=2t.
\]
They imply \(N_{10}=N_{01}\) and then \(N_{00}=N_{11}\).

By \cref{lem:Ca-properties}, the sets
\[
U
\quad\text{and}\quad
C_a\quad(a\in U\setminus\set{x,y})
\]
are \(n-1\) distinct members containing both \(x\) and \(y\).  Hence \(N_{11}\ge n-1\).
\end{proof}

\begin{remark}[Where the extra one goes]\label{rem:parity-firewall}
In \(\F^+=\F\cup\set{\varnothing}\), the critical marginals remain \(t\), but the \(00\)-cell gains the empty set.  Thus
\[
N_{00}^{+}=N_{00}+1=N_{11}+1.
\]
The identity with \(+1\) is correct only in the odd standard family \(\F^+\), not in the even family \(\F\) used below.
\end{remark}

\section{The height-four theorem}

For a critical pair \(x,y\), define the double-avoidance fiber
\begin{equation}\label{eq:Gxy}
\G_{xy}:=\set{A\in\F:x,y\notin A},
\qquad
T_{xy}:=\bigcup\G_{xy}.
\end{equation}
It is a union-closed subfamily and \(T_{xy}\in\G_{xy}\).

\begin{lemma}[Height-two extremal bound]\label{lem:height2-extremal}
Let \(\G\) be a finite union-closed family with ground set
\(
T=\bigcup\G
\), \(\abs{T}=r\), and \(H(\G)\le2\).  Then
\begin{equation}\label{eq:height2-size}
\abs{\G}\le r+1.
\end{equation}
If equality holds, then
\begin{equation}\label{eq:height2-equality}
\G=\set{T}\cup\set{T\setminus\set{a}:a\in T}.
\end{equation}
\end{lemma}

\begin{proof}
The set \(T\) is the maximum member.  Two distinct proper members cannot be comparable, since either one together with \(T\) would form a three-chain.  Their union must equal \(T\); otherwise either proper member, their union, and \(T\) form a three-chain.  Consequently the nonempty complements \(T\setminus A\), for proper \(A\in\G\), are pairwise disjoint.  There are at most \(r\) of them, proving \cref{eq:height2-size}.  In the equality case there are \(r\) pairwise disjoint nonempty subsets of an \(r\)-element set, so each is a singleton, giving \cref{eq:height2-equality}.
\end{proof}

\begin{proof}[Proof of \cref{thm:main-height4}]
Assume for contradiction that a counterexample of height at most four exists, and let \(\F\) be a \(4\)-minimal one.  By \cref{lem:critical-three,prop:lower-bound},
\begin{equation}\label{eq:h4-setup}
\abs{\F}=2t,
\qquad
\abs{K}\ge3,
\qquad
 t\ge2n-1.
\end{equation}
Fix distinct \(x,y\in K\).  By \cref{lem:pair-parity},
\begin{equation}\label{eq:Gxy-lower}
\abs{\G_{xy}}=N_{00}=N_{11}\ge n-1.
\end{equation}
Every \(A\in\G_{xy}\) satisfies
\[
A\subsetneq C_x\subsetneq U.
\]
Indeed, \(A\subseteq C_x\), the first containment is strict because \(y\in C_x\setminus A\), and the second is strict because \(x\notin C_x\).  A three-chain in \(\G_{xy}\) would therefore extend by \(C_x,U\) to a five-chain in \(\F\).  Hence
\begin{equation}\label{eq:Gxy-height2}
H(\G_{xy})\le2.
\end{equation}
Since \(T_{xy}\subseteq U\setminus\set{x,y}\), \cref{lem:height2-extremal} and \cref{eq:Gxy-lower} give
\[
n-1
\le \abs{\G_{xy}}
\le \abs{T_{xy}}+1
\le n-1.
\]
All inequalities are equalities.  Thus
\begin{equation}\label{eq:Gxy-exact}
T_{xy}=U\setminus\set{x,y},
\qquad
\G_{xy}
=
\set{T_{xy}}
\cup
\set{T_{xy}\setminus\set{a}:a\in T_{xy}}.
\end{equation}
If \(n=3\), \cref{eq:Gxy-exact} forces \(\varnothing\in\F\), already a contradiction.  Hence we may assume \(n\ge4\).

Choose three distinct critical elements \(x,y,z\).  Let
\[
M_{ijk}
=
\abs{\set{A\in\F:
\1_{x\in A}=i,\ \1_{y\in A}=j,\ \1_{z\in A}=k}},
\qquad i,j,k\in\set{0,1}.
\]
Within the \(xy\)-avoidance fiber in \cref{eq:Gxy-exact}, exactly one member, namely
\(
U\setminus\set{x,y,z}
\), avoids \(z\).  Therefore
\[
M_{000}=1,
\qquad
M_{001}=n-2.
\]
Applying the same statement to the pairs \(y,z\) and \(x,z\) yields
\begin{equation}\label{eq:lower-cube}
M_{100}=M_{010}=M_{001}=n-2.
\end{equation}

The \(xy\)-containing cell has exactly \(n-1\) members, and the \(n-1\) members listed in the proof of \cref{lem:pair-parity} exhaust it:
\[
U
\quad\text{and}\quad
C_a\quad(a\notin\set{x,y}).
\]
Exactly one of them, namely \(C_z\), avoids \(z\).  Hence
\[
M_{110}=1,
\qquad
M_{111}=n-2.
\]
By symmetry,
\begin{equation}\label{eq:upper-cube}
M_{101}=M_{011}=M_{110}=1.
\end{equation}
Equivalently, the eight cells are forced to be
\[
\begin{array}{c|cccc}
 & (x,y)=(0,0) & (1,0) & (0,1) & (1,1)\\
\midrule
z=0 & 1 & n-2 & n-2 & 1\\
z=1 & n-2 & 1 & 1 & n-2
\end{array}
\]
It follows that
\[
\begin{aligned}
f(x)
&=M_{100}+M_{101}+M_{110}+M_{111}\\
&=(n-2)+1+1+(n-2)=2n-2.
\end{aligned}
\]
Since \(x\) is critical, \(f(x)=t\), contradicting \(t\ge2n-1\) in \cref{eq:h4-setup}.  This proves the theorem.
\end{proof}

\section{Height five: critical coatoms and pair dichotomy}

For the remainder of the paper, suppose that a counterexample with height at most five exists, and let \(\F\) be a \(5\)-minimal counterexample.  We retain the notation
\begin{equation}\label{eq:h5-setup}
\abs{\F}=2t,
\qquad
f(a)\le t,
\qquad
K=\set{a:f(a)=t},
\qquad
\abs{K}\ge3,
\qquad
 t\ge2n-1.
\end{equation}
Define the deficit
\begin{equation}\label{eq:deficit}
\delta_a:=t-f(a)\ge0.
\end{equation}

For distinct critical elements \(x,y\), put
\begin{equation}\label{eq:d-r}
d_{xy}:=\abs{\G_{xy}}=N_{00}=N_{11},
\qquad
r_{xy}:=N_{10}=N_{01}=t-d_{xy}.
\end{equation}
Separation implies \(r_{xy}\ge1\).

\begin{theorem}[Critical coatoms]\label{thm:critical-coatom}
For every critical element \(x\in K\),
\begin{equation}\label{eq:critical-coatom}
\boxed{C_x=U\setminus\set{x}.}
\end{equation}
\end{theorem}

\begin{proof}
Fix \(y\in K\setminus\set{x}\).  As before,
\[
\abs{\G_{xy}}=d_{xy}\ge n-1,
\]
and every chain in \(\G_{xy}\) extends by \(C_x,U\), so
\begin{equation}\label{eq:Gxy-height3}
H(\G_{xy})\le3.
\end{equation}
If \(H(\G_{xy})\le2\), then \cref{lem:height2-extremal} gives
\[
n-1\le \abs{\G_{xy}}
\le \abs{T_{xy}}+1
\le n-1.
\]
Thus \(T_{xy}=U\setminus\set{x,y}\).  Since \(T_{xy}\subseteq C_x\), \(y\in C_x\), and \(x\notin C_x\), we obtain \(C_x=U\setminus\set{x}\).

It remains to consider \(H(\G_{xy})=3\).  Because \(T_{xy}\) is the maximum member of \(\G_{xy}\), there is a chain
\begin{equation}\label{eq:ABT}
A\subsetneq B\subsetneq T_{xy}.
\end{equation}
Indeed, the top of any three-chain must be \(T_{xy}\), or adjoining \(T_{xy}\) would create a four-chain in \(\G_{xy}\).

Let \(X\in\F\) contain \(x\).  If \(C_x\cup X\subsetneq U\), then
\[
A\subsetneq B\subsetneq T_{xy}
\subsetneq C_x
\subsetneq C_x\cup X
\subsetneq U
\]
is a six-chain: the containment \(T_{xy}\subsetneq C_x\) is strict because \(y\in C_x\setminus T_{xy}\), and \(C_x\subsetneq C_x\cup X\) is strict because \(x\in X\setminus C_x\).  Hence
\begin{equation}\label{eq:Cx-union-X}
C_x\cup X=U
\qquad\text{for every }X\in\F_x.
\end{equation}
If some \(q\ne x\) were absent from \(C_x\), then every member avoiding \(x\) would avoid \(q\), while \cref{eq:Cx-union-X} would force every member containing \(x\) to contain \(q\).  Thus \(x\) and \(q\) would have identical occurrence vectors, contradicting separation.  Therefore \(U\setminus C_x=\set{x}\).
\end{proof}

\begin{theorem}[Critical-pair dichotomy]\label{thm:pair-dichotomy}
For any distinct \(x,y\in K\), one of the following holds:
\begin{equation}\label{eq:dichotomy}
\boxed{
T_{xy}=U\setminus\set{x,y}
\quad\text{or}\quad
 d_{xy}\ge\left\lceil\frac{t+1}{2}\right\rceil.}
\end{equation}
In the second case, \(d_{xy}\ge n\).
\end{theorem}

\begin{proof}
Suppose \(T_{xy}\ne U\setminus\set{x,y}\).  The first part of the proof of \cref{thm:critical-coatom} shows that \(H(\G_{xy})\) cannot be at most two.  Hence there is a chain \(A\subsetneq B\subsetneq T_{xy}\) as in \cref{eq:ABT}.

Let \(X\) be any member containing \(x\) but not \(y\).  By \cref{thm:critical-coatom}, \(C_y=U\setminus\set{y}\).  The set \(T_{xy}\cup X\) is a member of \(\F\), avoids \(y\), and strictly contains \(T_{xy}\).  If it were a proper subset of \(C_y\), then
\[
A\subsetneq B\subsetneq T_{xy}
\subsetneq T_{xy}\cup X
\subsetneq C_y
\subsetneq U
\]
would be a six-chain.  Thus
\begin{equation}\label{eq:x-only-union}
T_{xy}\cup X=U\setminus\set{y}
\end{equation}
for every \(x\)-only member \(X\).  Symmetrically, every \(y\)-only member \(Y\) satisfies
\begin{equation}\label{eq:y-only-union}
T_{xy}\cup Y=U\setminus\set{x}.
\end{equation}

Choose
\(
q\in U\setminus(T_{xy}\cup\set{x,y})
\), which exists by assumption.  Equations \cref{eq:x-only-union,eq:y-only-union} show that \(q\) belongs to all \(r_{xy}\) \(x\)-only members and all \(r_{xy}\) \(y\)-only members; it also belongs to \(U\).  Therefore
\[
f(q)\ge 2r_{xy}+1=2(t-d_{xy})+1.
\]
Since \(f(q)\le t\), we obtain \(d_{xy}\ge(t+1)/2\), proving the second alternative.  Finally \(t\ge2n-1\) implies \(\lceil(t+1)/2\rceil\ge n\).
\end{proof}

\section{The trace normal form for a nonfull critical pair}

Fix a critical pair \(x,y\) in the second alternative of \cref{thm:pair-dichotomy}, and abbreviate
\begin{equation}\label{eq:T-P}
T:=T_{xy},
\qquad
P:=U\setminus T,
\qquad
k:=\abs{P}.
\end{equation}
Then \(x,y\in P\), and nonfullness gives \(k\ge3\).  Fix a chain
\begin{equation}\label{eq:fixed-ABT}
A\subsetneq B\subsetneq T
\end{equation}
in \(\G_{xy}\).

\begin{lemma}[Coatoms on the missing-coordinate set]\label{lem:P-coatoms}
For every \(p\in P\),
\begin{equation}\label{eq:P-coatom}
C_p=U\setminus\set{p}.
\end{equation}
\end{lemma}

\begin{proof}
For \(p=x\) or \(p=y\), this is \cref{thm:critical-coatom}.  Let \(p\in P\setminus\set{x,y}\).  By \cref{lem:Ca-properties}, both critical elements \(x,y\) belong to \(C_p\), whereas neither belongs to \(T\).  Thus \(T\subsetneq C_p\).

For any \(X\in\F_p\), if \(C_p\cup X\subsetneq U\), then
\[
A\subsetneq B\subsetneq T
\subsetneq C_p
\subsetneq C_p\cup X
\subsetneq U
\]
would be a six-chain.  Hence \(C_p\cup X=U\) for every member containing \(p\).  If a second coordinate \(q\ne p\) were absent from \(C_p\), then every member avoiding \(p\) would avoid \(q\), and every member containing \(p\) would contain \(q\).  This violates separation.  Therefore \(U\setminus C_p=\set{p}\).
\end{proof}

\begin{theorem}[Three-trace normal form]\label{thm:trace-normal}
For every \(F\in\F\), the trace of \(F\) on \(P\) is one of
\begin{equation}\label{eq:three-traces}
\boxed{
F\cap P\in
\set{\varnothing}
\cup
\set{P\setminus\set{p}:p\in P}
\cup
\set{P}.}
\end{equation}
\end{theorem}

\begin{proof}
Let \(R=F\cap P\).  There is nothing to prove if \(R=\varnothing\) or \(R=P\).  Suppose \(\varnothing\ne R\subsetneq P\), and choose \(p\in P\setminus R\).  The member \(T\cup F\) strictly contains \(T\), avoids \(p\), and therefore is contained in \(C_p=U\setminus\set{p}\).  If the containment were strict, then
\[
A\subsetneq B\subsetneq T
\subsetneq T\cup F
\subsetneq C_p
\subsetneq U
\]
would be a six-chain.  Hence \(T\cup F=C_p\).  Taking traces on \(P\) gives \(R=P\setminus\set{p}\).
\end{proof}

For a family \(\mathcal X\subseteq2^T\) and a fixed trace \(R\subseteq P\), write
\[
\mathcal X\star R:=\set{S\cup R:S\in\mathcal X}.
\]
Define families of subsets of \(T\) by
\begin{align*}
\G&:=\set{S\subseteq T:S\in\F},\\
\Hh_p&:=\set{S\subseteq T:S\cup(P\setminus\set{p})\in\F}
\qquad(p\in P),\\
\W&:=\set{S\subseteq T:S\cup P\in\F}.
\end{align*}
Then \cref{thm:trace-normal} gives the disjoint decomposition
\begin{equation}\label{eq:layer-decomp}
\F
=
(\G\star\varnothing)
\mathbin{\dot\cup}
\bigsqcup_{p\in P}(\Hh_p\star(P\setminus\set{p}))
\mathbin{\dot\cup}
(\W\star P).
\end{equation}
For families on \(T\), let
\(
\mathcal X\join\mathcal Y=
\set{X\cup Y:X\in\mathcal X,\ Y\in\mathcal Y}
\).

\begin{proposition}[Layer closure and counting identities]\label{prop:layer-count}
The decomposition \cref{eq:layer-decomp} satisfies
\begin{align}
\G\join\G&\subseteq\G,
&\G\join\Hh_p&\subseteq\Hh_p,
&\Hh_p\join\Hh_p&\subseteq\Hh_p,
\label{eq:layer-closure1}\\
\Hh_p\join\Hh_q&\subseteq\W &&(p\ne q),
&\G\join\W&\subseteq\W,
&\Hh_p\join\W&\subseteq\W,
\label{eq:layer-closure2}\\
\W\join\W&\subseteq\W.&&&&
\label{eq:layer-closure3}
\end{align}
These are exactly the tracewise conditions for union-closure.  Moreover,
\begin{equation}\label{eq:T-in-layers}
T\in\G\cap\bigcap_{p\in P}\Hh_p\cap\W.
\end{equation}
Let
\[
g=\abs{\G},
\qquad
h_p=\abs{\Hh_p},
\qquad
w=\abs{\W}.
\]
Then
\begin{equation}\label{eq:total-layer-count}
2t=g+\sum_{p\in P}h_p+w,
\end{equation}
and for every \(p\in P\),
\begin{equation}\label{eq:hp-deficit}
h_p=t-g+\delta_p.
\end{equation}
Consequently,
\begin{equation}\label{eq:w-exact}
\boxed{
w=(k-1)g-(k-2)t-\sum_{p\in P}\delta_p.}
\end{equation}
In particular,
\begin{equation}\label{eq:g-lower-exact}
\boxed{
g\ge
\left\lceil
\frac{(k-2)t+1+\sum_{p\in P}\delta_p}{k-1}
\right\rceil
\ge
\left\lceil
\frac{(k-2)t+1}{k-1}
\right\rceil.}
\end{equation}
Here \(g=d_{xy}\), and \(\delta_p=0\) precisely when \(p\) is critical.
\end{proposition}

\begin{proof}
The trace unions
\[
\varnothing\cup\varnothing=\varnothing,
\quad
\varnothing\cup(P\setminus\set{p})=P\setminus\set{p},
\quad
(P\setminus\set{p})\cup(P\setminus\set{q})=P\quad(p\ne q),
\]
together with the analogous idempotent and full-trace unions, give \cref{eq:layer-closure1,eq:layer-closure2,eq:layer-closure3}.  Conversely, these inclusions check every possible pair of the traces in \cref{eq:three-traces}.

The set \(T\) is the maximum of \(\G_{xy}\), so \(T\in\G\).  By \cref{lem:P-coatoms},
\(
C_p=T\cup(P\setminus\set{p})
\), hence \(T\in\Hh_p\); and \(U=T\cup P\) gives \(T\in\W\).  This proves \cref{eq:T-in-layers} and also \(w\ge1\).

Equation \cref{eq:total-layer-count} is immediate from the disjoint decomposition.  A coordinate \(p\in P\) is absent exactly in the \(\G\)-layer and the \(\Hh_p\)-layer.  Therefore
\[
f(p)=2t-g-h_p.
\]
Using \(\delta_p=t-f(p)\) gives \cref{eq:hp-deficit}.  Substitution into \cref{eq:total-layer-count} yields \cref{eq:w-exact}.  Finally \(w\ge1\) gives \cref{eq:g-lower-exact}.  The identity \(g=d_{xy}\) follows from the definition of \(\G\).
\end{proof}

\section{Coordinate projection and the exact height-five obstruction}

Fix a critical element \(p\in K\).  Define the coordinate-deletion projection
\begin{equation}\label{eq:projection}
\pi_p(\F):=\set{A\setminus\set{p}:A\in\F}.
\end{equation}

\begin{lemma}[Projection preserves the admissible class]\label{lem:projection-class}
A counterexample \(\F\) contains no singleton \(\set{p}\).  For critical \(p\), the family \(\pi_p(\F)\) is union-closed, contains no empty set, and satisfies
\begin{equation}\label{eq:projection-height}
H(\pi_p(\F))\le H(\F).
\end{equation}
\end{lemma}

\begin{proof}
If \(\set{p}\in\F\), then the map
\[
A\longmapsto A\cup\set{p}
\]
injects the members avoiding \(p\) into those containing \(p\).  Its image does not contain \(\set{p}\), because \(\varnothing\notin\F\).  Thus \(p\) occurs strictly more often than it is absent, contradicting that \(\F\) is a counterexample.

Coordinate deletion commutes with union, so \(\pi_p(\F)\) is union-closed.  Its empty set could arise only from \(\varnothing\) or \(\set{p}\), neither of which belongs to \(\F\).

Let
\(
B_1\subsetneq\cdots\subsetneq B_s
\)
be a chain in \(\pi_p(\F)\), and choose preimages \(A_i\in\F\) with \(A_i\setminus\set{p}=B_i\).  Put
\[
D_i=A_1\cup\cdots\cup A_i.
\]
Then \(D_i\in\F\) and
\[
D_i\setminus\set{p}=B_1\cup\cdots\cup B_i=B_i.
\]
The projected sets are strictly increasing, so
\(
D_1\subsetneq\cdots\subsetneq D_s
\), proving \cref{eq:projection-height}.
\end{proof}

Define the family of lower endpoints of \(p\)-matching edges by
\begin{equation}\label{eq:Jp}
\J_p
:=
\set{A\in\F:p\notin A,\ A\cup\set{p}\in\F}.
\end{equation}
Let
\begin{equation}\label{eq:cp-jp}
c_p:=\abs{\J_p},
\qquad
j_p(z):=\abs{\set{A\in\J_p:z\in A}}
\quad(z\ne p).
\end{equation}
By \cref{thm:critical-coatom},
\(
C_p=U\setminus\set{p}\in\J_p
\), since \(C_p\cup\set{p}=U\).  In particular,
\begin{equation}\label{eq:cp-positive}
c_p\ge1.
\end{equation}

\begin{proposition}[Exact projection counts]\label{prop:projection-counts}
For every critical \(p\),
\begin{equation}\label{eq:projection-size}
\abs{\pi_p(\F)}=2t-c_p,
\end{equation}
and for every \(z\ne p\),
\begin{equation}\label{eq:projection-frequency}
f_{\pi_p(\F)}(z)=f(z)-j_p(z)=t-\delta_z-j_p(z).
\end{equation}
Since \(\pi_p(\F)\) has fewer members than \(\F\), there exists \(z\ne p\) such that
\begin{equation}\label{eq:projection-witness}
\boxed{c_p-2j_p(z)\ge2\delta_z+1.}
\end{equation}
\end{proposition}

\begin{proof}
Every fiber of \(\pi_p\) has size one or two.  A two-element fiber is exactly a pair
\(
A,A\cup\set{p}
\)
with \(A\in\J_p\), so each such pair lowers the member count by one.  This proves \cref{eq:projection-size}.

For \(z\ne p\), collapsing a matched pair lowers the frequency of \(z\) by one precisely when the lower endpoint contains \(z\).  This gives \cref{eq:projection-frequency}.

By \cref{eq:cp-positive}, \(\pi_p(\F)\) has fewer than \(2t\) members.  It is therefore not a counterexample, by the minimality of \(\F\) and \cref{lem:projection-class}.  Hence some \(z\ne p\) satisfies
\[
t-\delta_z-j_p(z)
>
\frac{2t-c_p}{2}.
\]
Rearranging gives
\(
c_p-2j_p(z)>2\delta_z
\); integrality yields \cref{eq:projection-witness}.
\end{proof}

\begin{corollary}[Small matching fibers]\label{cor:cp-atleast3}
For every critical element \(p\),
\begin{equation}\label{eq:cp-atleast3}
\boxed{c_p\ge3.}
\end{equation}
If \(c_p=3\), then every witness \(z\) in \cref{eq:projection-witness} is critical and satisfies \(j_p(z)=1\).
\end{corollary}

\begin{proof}
The matched lower endpoint \(C_p=U\setminus\set{p}\) contains every \(z\ne p\), so \(j_p(z)\ge1\) for all \(z\ne p\).  If \(c_p\le2\), then
\(
c_p-2j_p(z)\le0
\)
for every \(z\ne p\), contradicting \cref{eq:projection-witness}.  Hence \(c_p\ge3\).

If \(c_p=3\), a witness satisfies
\[
3-2j_p(z)\ge2\delta_z+1.
\]
Since \(j_p(z)\ge1\), the only possibility is \(j_p(z)=1\) and \(\delta_z=0\).
\end{proof}

The preceding proposition identifies the exact form of the missing transfer statement.

\begin{problem}[Critical projection transfer]\label{prob:transfer}
Prove that every \(5\)-minimal counterexample has a critical element \(p\) such that
\begin{equation}\label{eq:transfer}
\boxed{
c_p-2j_p(z)\le2\delta_z
\qquad\text{for every }z\ne p.}
\end{equation}
\end{problem}

Indeed, \cref{eq:transfer} and \cref{eq:projection-frequency} would give
\[
f_{\pi_p(\F)}(z)
=t-\delta_z-j_p(z)
\le t-\frac{c_p}{2}
=\frac{\abs{\pi_p(\F)}}2
=\frac{2t-c_p}{2}
\]
for every surviving coordinate \(z\), so \(\pi_p(\F)\) would be a smaller counterexample, a contradiction.  Conversely, \cref{prop:projection-counts} shows that any putative minimal counterexample must violate \cref{eq:transfer} at every critical coordinate.

The nonfull trace normal form localizes this obstruction further.

\begin{proposition}[Projection inside a nonfull pair]\label{prop:projection-nonfull}
Assume the notation of \cref{eq:T-P,eq:layer-decomp}, and let \(p\in P\cap K\).  Then
\begin{equation}\label{eq:J-intersection}
\J_p\longleftrightarrow\Hh_p\cap\W,
\end{equation}
in the sense that a lower endpoint is
\(
S\cup(P\setminus\set{p})
\)
with \(S\in\Hh_p\cap\W\).  Consequently,
\begin{align}
 c_p&=\abs{\Hh_p\cap\W},
\label{eq:cp-intersection}\\
 j_p(z)&=\abs{\set{S\in\Hh_p\cap\W:z\in S}}
\qquad(z\in T),
\label{eq:jp-T}\\
 j_p(q)&=c_p
\qquad(q\in P\setminus\set{p}).
\label{eq:jp-P}
\end{align}
Every projection witness for \(p\) lies in \(T\).
\end{proposition}

\begin{proof}
A member avoiding \(p\) can be paired with its \(p\)-extension only when its \(P\)-trace is \(P\setminus\set{p}\) and the upper trace is \(P\).  The two members have the same \(T\)-part \(S\), which proves \cref{eq:J-intersection,eq:cp-intersection,eq:jp-T}.  Every \(q\in P\setminus\set{p}\) belongs to every lower endpoint, giving \cref{eq:jp-P}.  For such \(q\),
\[
c_p-2j_p(q)=-c_p\le2\delta_q,
\]
so it cannot satisfy the witness inequality.  Thus every witness lies in \(T\).
\end{proof}

\section{The critical join-cover number}\label{sec:critical-cover}

Let \(\mathfrak J(\F)\) denote the set of join-irreducible members of \(\F\).  Thus \(B\in\mathfrak J(\F)\) if
\[
B=A_1\cup A_2,
\qquad A_1,A_2\in\F,
\]
forces \(B=A_1\) or \(B=A_2\).  Every member of a finite union-closed family is the union of the join-irreducible members below it.  Moreover, deleting any collection of join-irreducible members leaves a union-closed family: a union of two undeleted members cannot be a deleted join-irreducible.

\begin{definition}[Critical join-cover number]\label{def:critical-cover}
For a \(5\)-minimal counterexample, define
\begin{equation}\label{eq:rho-c}
\rho_c(\F)
:=
\min\set{s:\ \exists B_1,\dots,B_s\in\mathfrak J(\F)
\text{ with }K\subseteq B_1\cup\cdots\cup B_s}.
\end{equation}
\end{definition}

\begin{proposition}[A priori bounds]\label{prop:rho-bounds}
For a \(5\)-minimal counterexample,
\[
3\le \rho_c(\F)\le5.
\]
If \(B_1,\dots,B_s\) is a minimal critical join-cover and
\(
V:=B_1\cup\cdots\cup B_s\ne U
\), then \(s\le4\).
\end{proposition}

\begin{proof}
If one or two join-irreducible members covered all critical coordinates, deleting them would decrease the frequency of every critical coordinate and leave every noncritical frequency below \(t\).  The remaining family would have respectively \(2t-1\) or \(2t-2\) members, would still be union-closed and of height at most five, and would still have no strict-majority coordinate.  This contradicts minimality.  Hence \(\rho_c(\F)\ge3\).

For a minimal cover \(B_1,\dots,B_s\), each \(B_i\) contains a private critical coordinate
\[
x_i\in B_i\setminus\bigcup_{j\ne i}B_j.
\]
Consequently, after any ordering of the cover,
\[
B_1
\subsetneq B_1\cup B_2
\subsetneq\cdots\subsetneq
B_1\cup\cdots\cup B_s=V
\]
is an \(s\)-member chain.  Thus \(s\le5\).  If \(V\ne U\), adjoining \(U\) gives an \((s+1)\)-member chain, so \(s\le4\).
\end{proof}

The upper value in \cref{prop:rho-bounds} is impossible.  The proof proceeds by separating the Boolean tops generated by five private critical coordinates from the residual members.  The residual private supports form a regular multihypergraph on five vertices.  Height two in every two-support fiber then supplies enough rigidity to contradict the frequency capacity of the nonprivate coordinates.

\begin{theorem}[Five-cover exclusion]\label{thm:five-cover-exclusion}
A \(5\)-minimal counterexample cannot satisfy \(\rho_c(\F)=5\).  Consequently,
\begin{equation}\label{eq:rho-34}
\boxed{\rho_c(\F)\in\set{3,4}.}
\end{equation}
\end{theorem}

\subsection{Boolean fibers under a hypothetical five-cover}

Assume for contradiction that \(\rho_c(\F)=5\), and let \(B_1,\dots,B_5\) be a minimal critical join-cover.  By \cref{prop:rho-bounds},
\begin{equation}\label{eq:five-cover-U}
B_1\cup\cdots\cup B_5=U.
\end{equation}
Choose private critical coordinates
\begin{equation}\label{eq:private-xi}
x_i\in B_i\setminus\bigcup_{j\ne i}B_j
\qquad(i\in[5]),
\end{equation}
put \(X=\set{x_1,\dots,x_5}\), and for every nonempty \(I\subseteq[5]\) set
\begin{equation}\label{eq:SI}
S_I:=\bigcup_{i\in I}B_i.
\end{equation}
The \(31\) members \(S_I\) are pairwise distinct, since \(x_i\in S_I\) exactly when \(i\in I\).
For \(A\in\F\), define its private support
\[
I(A):=\set{i\in[5]:x_i\in A}.
\]

\begin{lemma}[Saturated Boolean fibers]\label{lem:boolean-fibers}
For every \(A\in\F\),
\begin{equation}\label{eq:saturated-containment}
I(A)\ne\varnothing,
\qquad
A\subseteq S_{I(A)}.
\end{equation}
For nonempty \(I\subseteq[5]\), the fiber
\[
\F_I:=\set{A\in\F:I(A)=I}
\]
has maximum member \(S_I\) and height at most \(\abs I\).  In particular,
\begin{equation}\label{eq:singleton-fiber}
\F_{\set{i}}=\set{B_i}.
\end{equation}
\end{lemma}

\begin{proof}
If \(I(A)=\varnothing\), then
\[
A
\subsetneq A\cup B_1
\subsetneq\cdots\subsetneq
A\cup B_1\cup\cdots\cup B_5=U
\]
is a six-member chain, each step adding a private coordinate.

Let \(I=I(A)\ne\varnothing\).  If \(A\nsubseteq S_I\), order the members of \(I\), take the chain of their partial unions up to \(S_I\), insert \(A\cup S_I\), and then adjoin the remaining \(B_j\)'s one at a time.  This again gives six members.  Hence \(A\subseteq S_I\), so \(S_I\) is the maximum of \(\F_I\).

A chain of \(\abs I+1\) members in \(\F_I\), followed by the \(5-\abs I\) missing basis members, would have six members.  Therefore \(H(\F_I)\le\abs I\).  When \(\abs I=1\), the fiber consists only of its maximum, proving \cref{eq:singleton-fiber}.
\end{proof}

Put
\begin{equation}\label{eq:R0}
R_0:=U\setminus X,
\qquad R:=\abs{R_0}.
\end{equation}
One has \(R\ge1\).  Indeed, if \(R=0\), then \(B_i=\set{x_i}\), and \cref{lem:boolean-fibers} forces \(\F\) to consist of the \(31\) nonempty subsets of \(X\), contrary to the even cardinality of a minimal counterexample.
For \(u\in R_0\), define its basis support
\begin{equation}\label{eq:sigma-u}
\sigma(u):=\set{i\in[5]:u\in B_i}.
\end{equation}

\begin{lemma}[Nonprivate support]\label{lem:support-atleast2}
For every \(u\in R_0\), one has \(\abs{\sigma(u)}\ge2\).
\end{lemma}

\begin{proof}
For each \(i\), the critical coatom \(C_{x_i}=U\setminus\set{x_i}\) has private support \([5]\setminus\set{i}\).  By \cref{lem:boolean-fibers},
\(
C_{x_i}\subseteq S_{[5]\setminus\set{i}}
\), while the reverse inclusion holds because \(S_{[5]\setminus\set{i}}\) avoids \(x_i\).  Hence
\begin{equation}\label{eq:coatom-top}
S_{[5]\setminus\set{i}}=U\setminus\set{x_i}.
\end{equation}
If \(\sigma(u)=\set{i}\), then \(u\) belongs to the right-hand side of \cref{eq:coatom-top} but not to the left-hand side.  The support cannot be empty because the \(B_i\)'s cover \(U\).
\end{proof}

\subsection{The residual support multihypergraph}

Put
\[
m_I:=\abs{\F_I},
\qquad q_I:=m_I-1,
\]
so that \(q_I\) counts the members of the fiber other than its Boolean top \(S_I\).  Let
\begin{equation}\label{eq:Q-d}
Q:=\sum_{\varnothing\ne I\subseteq[5]}q_I
=\abs\F-31
=2d+1,
\qquad d:=t-16.
\end{equation}
Every private critical coordinate \(x_i\) occurs in \(16\) Boolean tops, and hence
\begin{equation}\label{eq:residual-regular}
\sum_{I\ni i}q_I=d
\qquad(i\in[5]).
\end{equation}
By \cref{eq:singleton-fiber}, \(q_{\set{i}}=0\), so every residual private support has size at least two.

For \(u\in R_0\), let \(s=\abs{\sigma(u)}\).  Among the \(31\) Boolean tops, \(u\) appears in
\begin{equation}\label{eq:bs}
b_s=32-2^{5-s}
\end{equation}
members.  If \(e(u)\) denotes its number of occurrences among the \(Q\) residual members, then \(f(u)\le t=16+d\) gives
\begin{equation}\label{eq:e-upper}
e(u)\le d-\beta_s,
\qquad
\beta_s:=16-2^{5-s}.
\end{equation}
Thus
\[
\begin{array}{c|cccc}
 s&2&3&4&5\\ \hline
 b_s&24&28&30&31\\
 \beta_s&8&12&14&15
\end{array}
\]

For every \(u\in R_0\), the member \(C_u\) contains all private critical coordinates by \cref{lem:Ca-properties}; it is proper because it avoids \(u\).  Hence \(C_u\) is a residual member of full private support.  The \(R\) members \(C_u\) are distinct.

Remove these \(R\) distinguished full-support residual members.  The remaining \(Q-R\) residual members have total private incidence \(5d-5R\) and support size at least two.  Their total excess over size two is
\begin{equation}\label{eq:E}
E:=5d-5R-2(Q-R)=d-3R-2\ge0.
\end{equation}
Let
\begin{equation}\label{eq:P-pair}
P:=\sum_{\abs I=2}q_I
\end{equation}
be the number of two-support residual members.  At most \(E\) of the \(Q-R\) remaining members can have support at least three, and therefore
\begin{equation}\label{eq:P-lower}
P\ge Q-R-E=d+2R+3.
\end{equation}

\subsection{Rigidity of the two-support fibers}

Fix a two-set \(J\subseteq[5]\).  The fiber \(\F_J\) has height at most two and top \(S_J\).  The complements in \(S_J\) of its \(q_J\) residual members are pairwise disjoint nonempty subsets consisting only of elements of \(R_0\).  If
\[
N_J:=\abs{S_J\cap R_0},
\]
then
\begin{equation}\label{eq:qJ-NJ}
q_J\le N_J.
\end{equation}
Moreover, the total number of incidences between \(R_0\) and the residual members of this fiber is at least
\begin{equation}\label{eq:qJ-incidence}
q_J(q_J-1),
\end{equation}
because each \(u\in S_J\cap R_0\) is omitted by at most one residual member.

\begin{lemma}[Two-support rigidity]\label{lem:two-support-rigidity}
For every \(u\in R_0\),
\begin{equation}\label{eq:sigma-two}
\abs{\sigma(u)}=2.
\end{equation}
Moreover,
\begin{equation}\label{eq:R-atleast7}
R\ge7.
\end{equation}
\end{lemma}

\begin{proof}
Let \(s=\abs{\sigma(u)}\) and
\[
\gamma_s:=\binom{5-s}{2},
\]
the number of two-support types disjoint from \(\sigma(u)\).  In a disjoint fiber the top omits \(u\), so \cref{eq:qJ-NJ} gives at most \(R-1\) residual members.  In each of the other \(10-\gamma_s\) fibers, \(u\) is omitted by at most one residual member.  Thus the number of two-support residual members containing \(u\) is at least
\begin{equation}\label{eq:e-pair-lower}
P-\gamma_s(R-1)-(10-\gamma_s)
\ge d-7+2R-\gamma_s(R-2).
\end{equation}
This is a lower bound for \(e(u)\).

If \(s=3\), then \(\gamma_s=1\), and \cref{eq:e-pair-lower} is \(d+R-5>d-12\), contradicting \cref{eq:e-upper}.  If \(s=4\) or \(5\), then \(\gamma_s=0\), and the lower bound \(d+2R-7\) contradicts the corresponding upper bound in \cref{eq:e-upper}.  Together with \cref{lem:support-atleast2}, this proves \(s=2\).

For \(s=2\), one has \(\gamma_s=3\), so \cref{eq:e-pair-lower} gives \(e(u)\ge d-R-1\).  Comparing with \(e(u)\le d-8\) yields \(R\ge7\).
\end{proof}

\subsection{A lower bound for the residual degree}

\begin{lemma}\label{lem:d-lower}
Under the five-cover assumption,
\begin{equation}\label{eq:d-lower}
d\ge4R+4.
\end{equation}
\end{lemma}

\begin{proof}
Choose \(u\in R_0\) of maximum residual frequency \(e(u)\), and write \(\sigma(u)=\set{a,b}\).  For every \(v\in R_0\setminus\set{u}\), one has \(u\in C_v\).  Otherwise \(\F_u\subseteq\F_v\); all nonprivate coordinates have the same Boolean-top frequency \(24\), and maximality of \(e(u)\) forces equality, contradicting separation.

Let \(\mathcal R_{\ge3}\) be the residual members of private support at least three after the distinguished members \(C_v\) have been removed, and put \(h:=\abs{\mathcal R_{\ge3}}\).  For \(i\in[5]\), let \(H_i\) count the members of \(\mathcal R_{\ge3}\) whose support contains \(i\).  Since each such member contributes at least one unit to the excess in \cref{eq:E},
\begin{equation}\label{eq:h-E}
h\le E.
\end{equation}
The residual degree of private vertex \(i\) is \(d\).  Of these occurrences, \(R\) come from the \(C_v\)'s and \(H_i\) from \(\mathcal R_{\ge3}\); hence its degree among two-support residual members is \(d-R-H_i\).

Let \(q_{ab}\) be the number of residual members with private support \(\set{a,b}\).  The number of two-support residual members whose support meets \(\set{a,b}\) is
\[
2(d-R)-H_a-H_b-q_{ab}.
\]
They lie in seven two-support fibers, and in each such fiber at most one member omits \(u\).  Adding the \(R-1\) incidences \(u\in C_v\) for \(v\ne u\), we obtain
\[
e(u)\ge2d-R-H_a-H_b-q_{ab}-8.
\]
Since \(e(u)\le d-8\), it follows that
\begin{equation}\label{eq:Hqab}
H_a+H_b+q_{ab}\ge d-R.
\end{equation}
By \cref{eq:qJ-NJ}, \(q_{ab}\le R\).  Therefore
\[
d-2R
\le H_a+H_b
\le2h
\le2E
=2(d-3R-2),
\]
which is equivalent to \cref{eq:d-lower}.
\end{proof}

\subsection{The convexity contradiction}

\begin{proof}[Proof of \cref{thm:five-cover-exclusion}]
Let \(q_1,\dots,q_{10}\) be the residual multiplicities of the ten two-support fibers, so that \(P=\sum_{\nu=1}^{10}q_\nu\).  By \cref{eq:qJ-incidence}, their total incidence with \(R_0\) is at least
\begin{equation}\label{eq:pair-incidence-lower}
\sum_{\nu=1}^{10}q_\nu(q_\nu-1).
\end{equation}
On the other hand, \cref{eq:e-upper,eq:sigma-two} give total residual nonprivate incidence at most \(R(d-8)\).

The distinguished members \(C_v\) already contribute at least \(\binom R2\) nonprivate incidences.  Indeed, for each unordered pair \(\set{u,v}\subseteq R_0\), assume \(f(u)\ge f(v)\).  If \(u\notin C_v\), then \(\F_u\subseteq\F_v\), and the frequency inequality forces equality and inseparability.  Hence \(u\in C_v\).  Consequently,
\begin{equation}\label{eq:pair-incidence-upper}
\sum_{\nu=1}^{10}q_\nu(q_\nu-1)
\le
R(d-8)-\frac{R(R-1)}2.
\end{equation}
Cauchy--Schwarz gives
\begin{equation}\label{eq:cauchy-q}
\sum_{\nu=1}^{10}q_\nu(q_\nu-1)
\ge
\frac{P^2}{10}-P.
\end{equation}
By \cref{eq:P-lower}, \(P\ge d+2R+3\).  In the present range the function \(P^2/10-P\) is increasing.  Thus \cref{eq:pair-incidence-upper,eq:cauchy-q} require
\begin{equation}\label{eq:convexity-required}
\frac{(d+2R+3)^2}{10}-(d+2R+3)
\le
R(d-8)-\frac{R(R-1)}2.
\end{equation}
The left-hand side minus the right-hand side is
\[
\frac{d^2-6Rd+9R^2-4d+67R-21}{10}.
\]
Write \(d=4R+4+\eta\) with \(\eta\ge0\), using \cref{eq:d-lower}.  The numerator becomes
\[
R^2+2R\eta+59R+\eta^2+4\eta-21,
\]
which is positive for \(R\ge7\).  This contradicts \cref{eq:convexity-required}.  Therefore \(\rho_c(\F)\ne5\), and \cref{prop:rho-bounds} completes the proof.
\end{proof}

\subsection{The nonfull four-cover branch}

\begin{theorem}[One-coordinate normal form]\label{thm:four-cover-normal}
Suppose \(\rho_c(\F)=4\), and let \(B_1,\dots,B_4\) be a minimal critical join-cover with
\[
V:=B_1\cup\cdots\cup B_4\ne U.
\]
Then
\begin{equation}\label{eq:outside-one}
U\setminus V=\set{w}
\end{equation}
for a noncritical coordinate \(w\).  Moreover,
\begin{equation}\label{eq:stable-decomp}
\F
=\A\mathbin{\dot\cup}
\set{B\cup\set{w}:B\in\mathcal B},
\end{equation}
where \(\A,\mathcal B\subseteq2^V\) satisfy
\begin{equation}\label{eq:stable-closure}
\A\join\A\subseteq\A,
\qquad
\mathcal B\join\mathcal B\subseteq\mathcal B,
\qquad
\A\join\mathcal B\subseteq\mathcal B,
\end{equation}
and
\begin{equation}\label{eq:stable-size}
\abs\A>\abs{\mathcal B}.
\end{equation}
The lower fiber \(\A\) contains the \(15\) distinct Boolean-top members generated by \(B_1,\dots,B_4\); its private-support fiber over \(I\subseteq[4]\) has height at most \(\abs I\).
\end{theorem}

\begin{proof}
Choose private critical coordinates
\(
x_i\in B_i\setminus\bigcup_{j\ne i}B_j
\).
If \(A\nsubseteq V\) and \(A\cup V\subsetneq U\), then
\[
B_1
\subsetneq B_1\cup B_2
\subsetneq B_1\cup B_2\cup B_3
\subsetneq V
\subsetneq A\cup V
\subsetneq U
\]
is a six-member chain.  Hence every member not contained in \(V\) contains all of \(U\setminus V\).  Members contained in \(V\) omit all of \(U\setminus V\), so the coordinates outside \(V\) have identical occurrence vectors.  Separation gives \cref{eq:outside-one}.

Since the four basis members cover every critical coordinate, \(w\) is noncritical.  Splitting the family according to the presence of \(w\) gives \cref{eq:stable-decomp}; union-closure gives \cref{eq:stable-closure}.  Since
\[
f(w)=\abs{\mathcal B}<t
=\frac{\abs\A+\abs{\mathcal B}}2,
\]
we obtain \cref{eq:stable-size}.  The Boolean-top and fiber-height assertions follow from the same six-chain argument as in \cref{lem:boolean-fibers}, now using \(U\) as the final sixth member.
\end{proof}

\section{The remaining height-five branches}\label{sec:remaining}

Combining the preceding results, every \(5\)-minimal counterexample must lie in one of the following branches:
\begin{enumerate}[label=\textup{(\Alph*)},leftmargin=2.6em]
\item \(\rho_c(\F)=3\);
\item \(\rho_c(\F)=4\) and a minimal four-cover has union \(U\);
\item \(\rho_c(\F)=4\) and \(U=V\mathbin{\dot\cup}\set{w}\), with the stable two-fiber normal form of \cref{thm:four-cover-normal}.
\end{enumerate}
In every branch, the critical coatoms, the critical-pair dichotomy, and the projection-witness inequalities remain in force.

In branch \textup{(C)}, write
\(
a=\abs\A>b=\abs{\mathcal B}
\)
and define
\[
d_{\A}(y):=\abs{\set{A\in\A:y\notin A}}.
\]
A sufficient statement is the stable-pair defect inequality
\begin{equation}\label{eq:stable-defect-target}
d_{\A}(y)-f_{\mathcal B}(y)
<\frac{a-b}{2}
\end{equation}
for some \(y\in V\).  Such a coordinate would satisfy
\[
f_{\A}(y)+f_{\mathcal B}(y)
>\frac{a+b}{2},
\]
contradicting the counterexample assumption.

The projection formulation gives a second coordinatewise target: find a critical \(p\) satisfying \cref{eq:transfer}.  The main unresolved issue is the simultaneous compatibility of the stable trace geometry, the critical join-cover constraints, and the matching-defect witnesses.

\begin{problem}[Stable-pair defect transfer]\label{prob:stable-defect}
Prove \cref{eq:stable-defect-target} for every stable pair \((\A,\mathcal B)\) arising from branch \textup{(C)}.
\end{problem}

\section{Concluding remarks}

The height-four proof closes through a rigid equality case.  Two critical coordinates force a double-avoidance fiber with at least \(n-1\) members; height four restricts that fiber to height two, where the extremal family is unique.  A third critical coordinate then determines the eight incidence cells and gives \(t=2n-2\), contradicting the linear lower bound \(t\ge2n-1\).

At height five, the additional level has been reduced substantially but not eliminated.  Critical coordinates are genuine coatoms; nonfull critical pairs have a three-trace geometry; critical-coordinate projections impose exact matching-defect witnesses; and the critical join-cover number is only three or four.  The five-cover exclusion is global: a hypothetical five-cover generates a saturated copy of the nonempty Boolean lattice on five private critical coordinates, while the residual support multihypergraph is incompatible with the frequency capacity of the nonprivate coordinates.

Thus the height-five problem is reduced to explicit three- and four-cover configurations rather than an undifferentiated general case.  Completing the argument requires either excluding those remaining configurations or proving a projection-transfer inequality valid across them.  No unrestricted empty-set-free height-five theorem is claimed here.

\end{document}